\documentclass[letterpaper, 10 pt, journal, twoside]{ieeetran}
\usepackage[utf8]{inputenc}

\usepackage[normalem]{ulem}
\usepackage{cite}
\usepackage{amsmath,amssymb,amsfonts,mathtools}
\usepackage{graphicx}
\usepackage[export]{adjustbox}
\usepackage[dvipsnames]{xcolor}
\usepackage{hyperref}
\usepackage{cancel}
\usepackage{caption}
\usepackage{arydshln}
\usepackage{epsfig}
\usepackage{verbatim} 
\usepackage{tikz}

\usepackage{amsthm}

\newcommand{\continue}{\\ &\;\;}
\newcommand{\norm}[1]{\left|\left| #1 \right|\right|}

\newcommand{\reals}{\mathbb{R}}
\newcommand{\transpose}{^\textrm{T}}
\newcommand{\continuous}{\mathcal{C}}
\DeclareMathOperator*{\argmin}{arg\,min}

\DeclareMathOperator*{\wrap}{wrap}

\newtheorem{theorem}{}

\newtheorem{lemma}{}

\newtheorem{definition}{}

\newtheorem{corollary}{}

\newtheorem{remark}{}

\newtheorem{problem}{}

\def\BibTeX{{\rm B\kern-.05em{\sc i\kern-.025em b}\kern-.08em
    T\kern-.1667em\lower.7ex\hbox{E}\kern-.125emX}}

\title{Control Barrier Functions in Sampled-Data Systems}

\author{Joseph Breeden, Kunal Garg, and Dimitra Panagou
\thanks{This work was supported by the Air Force Office of Scientific Research under award number FA9550-17-1-0284 and the National Science Foundation under award number 1942907 and the Graduate Research Fellowship Program.}%
\thanks{The authors are with the Department of Aerospace Engineering, University of Michigan, Ann Arbor, MI, USA. Email: \texttt{\{jbreeden,kgarg,dpanagou\}@umich.edu}%
}
}

\definecolor{subsectioncolor}{cmyk}{0, 0, 0, 1}

\begin{document}

\maketitle
\thispagestyle{empty}

\begin{abstract}
This paper presents conditions for ensuring forward invariance of safe sets under sampled-data system dynamics with piecewise-constant controllers and fixed time-steps. First, we introduce two different metrics to compare the conservativeness of sufficient conditions on forward invariance under piecewise-constant controllers. Then, we propose three approaches for guaranteeing forward invariance, two motivated by continuous-time barrier functions, and one motivated by discrete-time barrier functions. All proposed conditions are control affine, and thus can be incorporated into quadratic programs for control synthesis. We show that the proposed conditions are less conservative than those in earlier studies, {\color{black} and show via simulation how this enables the use of barrier functions that are impossible to implement with the desired time-step using existing methods.}

\end{abstract}


\section{Introduction}

\IEEEPARstart{C}{ontrol} barrier functions (CBFs) and quadratic programs (QPs) have recently gained popularity for safety-critical control applications across disciplines, including vehicle control \cite{Vehicles_Journal,under_review_discrete_CBFs_vehicles}, bipedal robots \cite{Bipedal_Robots,Discrete_CBFs}, mechanical hands \cite{Cortez}, and multi-agent systems \cite{MultiRobots_Journal}. CBF conditions apply to {\color{black}both} continuous-time \cite{Vehicles_Journal,Bipedal_Robots,MultiRobots_Journal,Cortez} and discrete-time \cite{Discrete_CBFs,under_review_Discrete_MPC_CBFs,under_review_discrete_CBFs_vehicles} systems. In practice, physical systems evolve in continuous time under controllers that are implemented in discrete time, such as zero-order-hold (ZOH) controllers with fixed time-step. One can easily construct counter-examples showing that the control laws developed from the CBF condition in \cite{Vehicles_Journal,Bipedal_Robots,MultiRobots_Journal} are no longer safe when the controller is executed in discrete steps. On the other hand, a controller implemented under discrete-time CBFs may not satisfy the continuous safety condition between time steps \cite{CBF_Planning}. 

Recently, \cite{Cortez} proposed a method for ensuring satisfaction of the continuous-time CBF condition using a {\color{black}ZOH} control law by bounding the time derivative of the CBF between time steps. The {\color{black}method is extended} in \cite{James} to multi-agent systems in the presence of adversaries and uncertainty. The authors in \cite{Singletary} propose a similarly motivated approach, which also addresses uncertainty and input delay, using reachable set theory. 
{\color{black}In all cases,} certain safe states might be cast unreachable, or excessive control inputs might be used to avoid unsafe regions.

This paper studies conditions for forward invariance of safe sets under ZOH controllers. We begin by defining two types of margins, the \textit{controller margin} and the \textit{physical margin}, to compare the conservatism of the conditions developed. We then present extensions to the approaches in \cite{Cortez,James,Singletary} that reduce conservatism as measured by these margins, while similarly relying on proving that the continuous-time CBF condition is always satisfied. We then approach the problem starting from discrete-time CBF conditions such as in \cite{Discrete_CBFs,Blanchini}, and develop novel sufficient conditions on the forward invariance of a safe set under ZOH controllers.
Finally, we present simulations using the existing and new conditions on an obstacle-avoidance problem for a unicycle agent, and on a spacecraft attitude-control problem. The simulations demonstrate how the reduced conservatism of the proposed approaches enables {\color{black} both the achievement of tight tolerance mission objectives and the ZOH application of CBFs under time-steps that were not possible using the method in \cite{Cortez,James}.} 

\section{Preliminaries and Problem Formulation} \label{sec:preliminaries}

\textbf{Notations:} Let $\continuous^r$ be the set of $r$-times continuously differentiable functions, and let $\mathcal{C}^r_{loc}$ be the subset of $\mathcal{C}^r$ with locally Lipschitz $r$th derivatives. A function $\alpha:\reals\rightarrow\reals$ is extended class-$\mathcal{K}$, denoted $\alpha \in \mathcal{K}$, if it is continuous, strictly increasing, and $\alpha(0) = 0$. $B_r(x)$ denotes the closed ball centered at $x$ of radius $r$. $||\cdot||$ refers to the 2-norm when $(\cdot)$ is a vector, and the matrix-induced 2-norm when $(\cdot)$ is a matrix. $||\cdot||_\infty$ refers to the infinity-norm. {\color{black} Let $\cdot$ refer to the inner product, and $\times$ refer to the vector product}. $\wrap_\pi(\lambda)$ wraps $\lambda$ to $[-\pi, \pi]$. Let $\nabla [h]$ denote the gradient of $h$. Let $L_f h(x)$ denote the Lie derivative of $h$ along $f$ at $x$, $L_f h(x) = \nabla[h(x)]f(x)$. For a given dynamical system, let $\mathcal{R}(x(0),T)$ denote the set of states reachable from some $x(0)\in \reals^n$ in times $0\leq t< T$.

\vspace{2pt}
\noindent \textbf{Problem formulation:} We consider the system
\begin{equation}
	\dot{x} = f(x) + g(x) u \,, \label{eq:model}
\end{equation}
with state $x \in \reals^n$, control input $u\in U \subset \reals^m$ where $U$ is compact, and locally Lipschitz continuous functions $f: \reals^n \rightarrow \reals^n$ and $g: \reals^n \rightarrow \reals^{n \times m}$. Define $u_\textrm{max} \triangleq \max_{u\in U} ||u||$. Let $h: \reals^n \rightarrow \reals$ where $h \in \continuous^1_{loc}$, and define a safe set $S$ as 
\begin{equation}
	S \triangleq \{ x \in \reals^n \mid h(x) \leq 0\} \,. \label{eq:safe_set}
\end{equation}

For a continuous control law $u(x)$, the problem of rendering $S$ forward invariant is solved in \cite{Vehicles_Journal} using Zeroing CBFs, in the sequel called simply CBFs. This leads to a condition of the following form, adapted to the notation of this paper.

\begin{lemma}[{\hspace{-0.2pt}\cite[Cor. 2]{Vehicles_Journal}}] \label{prior:cbf_condition}
    Let $\alpha\in \mathcal{K}$. Let $h:\reals^n\rightarrow\reals$, $h\in \continuous^1$ define a set $S = \{x\in\reals^n : h(x)\leq 0\}$. Then for the system \eqref{eq:model}, any Lipschitz continuous control input $u(x)$ satisfying
    \begin{equation}
        L_f h(x) + L_g h(x) u(x) \leq \alpha(-h(x)), {\color{black} \; \forall x \in S,}
        \label{eq:cbf_condition}
    \end{equation}
renders $S$ forward invariant along the closed-loop trajectories of \eqref{eq:model}.
\end{lemma}

To apply \ref{prior:cbf_condition}, we must ensure \eqref{eq:cbf_condition} is satisfied along $x(t)$ for all $t \geq 0$.
However, suppose instead that the state $x$ is only measured discretely (and thus $u(x)$ is updated discretely too) at times $t_k = kT, k=0,1,2,\cdots$ for a fixed time-step $T\in\reals_{>0}$. 
Consider a ZOH control law
\begin{equation}
    u(t) = u_k,\; \forall t \in [t_k, t_{k+1}) \,, \label{eq:piecewise_control}
\end{equation}
where $u_k = u_k(x_k)\in U$ and $x_k = x(t_k)$, $\forall k\in \mathbb N$\footnote{Under $u$ as in \eqref{eq:piecewise_control} for a compact set $U$, uniqueness of the maximal closed-loop solution $x(t)$ (and hence $x_k$) is guaranteed by \cite[Thm. 54]{uniqueness}.}. 
Satisfaction of \eqref{eq:cbf_condition} only discretely is not sufficient for safety.
Thus, we seek a condition similar to \eqref{eq:cbf_condition} under which safety can be guaranteed when the control input is updated only at discrete times. To this end, we consider the following problem.

\begin{problem} \label{problem:design}
    Design a function $\phi:\reals_{>0}\times\reals^n \rightarrow \reals$ such that any bounded, piecewise-constant control input $u\in U$ of the form \eqref{eq:piecewise_control} satisfying
    \begin{equation}
        L_f h(x_k) + L_g h(x_k) u_k \leq \phi(T, x_k), \label{eq:zoh_cbf_condition}
    \end{equation}
    at the sampled states $x_k=x(kT), k\in\mathbb N$ renders $S$ forward invariant along the closed-loop trajectories of \eqref{eq:model}.
\end{problem}
\noindent We call \eqref{eq:zoh_cbf_condition} the ZOH-CBF condition. 
The following result, adapted from \cite{Cortez}, provides one form of the function $\phi$ that solves \ref{problem:design} (see also \cite{James}).

\begin{lemma}[{\hspace{-0.2pt}\cite[Thm. 2]{Cortez}}]
	\label{prior:zoh}
	Let the set $S$ in \eqref{eq:safe_set} be compact and $\alpha\in\mathcal{K}$ be locally Lipschitz continuous. Let $l_{L_fh}, l_{L_gh}, l_{\alpha(h)}$ be the Lipschitz constants of $L_fh, L_gh, \alpha(-h)$, respectively. Then the function $\phi_0^g:\reals_{>0}\times\reals^n$, defined as
	\begin{equation}
		\phi_0^g(T,x) \triangleq \alpha(-h(x)) - \frac{l_1 \Delta}{l_2}\left( e^{l_2 T} - 1\right) \,, \label{eq:phi0}
	\end{equation}
	solves \ref{problem:design}, 	where $l_1 = l_{L_f h} + l_{L_g h}u_\textrm{max} + l_{\alpha(h)}, l_2 = l_{L_f h} + l_{L_g h}u_\textrm{max}$, and $\Delta = \sup_{x\in S,u\in U}||f(x) + g(x)u||$.
\end{lemma}

Note that \eqref{eq:cbf_condition} and \eqref{eq:zoh_cbf_condition} are sufficient, not necessary, conditions for forward invariance \cite[Rem. 12]{CBFs_FirstUse}. 
In practice, the form of the function $\phi_0^g$ in \eqref{eq:phi0} is conservative in the sense that many safe trajectories may fail to satisfy \eqref{eq:zoh_cbf_condition} for $\phi=\phi_0^g$, as illustrated in Section~\ref{sec:results}. The work in this paper is devoted to developing alternative solutions to \ref{problem:design} that are less conservative compared to \eqref{eq:phi0}. We first introduce two metrics to quantify the conservatism of solutions to \ref{problem:design}.

\vspace{2pt}
\noindent \textbf{Comparison metrics:} We consider functions $\phi$ of the form:
\begin{equation}
    \phi
    (T,x)= \alpha(-h(x)) - \nu(T,x) \,, \label{eq:phi_form}
\end{equation}
where $\alpha$ is a class-$\mathcal{K}$ function that vanishes as $h(x)\rightarrow 0$, and $\nu:\reals_{>0}\times\reals^n\rightarrow\reals$ is a function of the discretization time-step $T$ and the state $x$ 
that does not explicitly depend on $h$. This motivates our first metric of comparison, defined as follows.

\begin{definition}[\textbf{Controller margin}] \label{def:control_margin}
    The function $\nu$ in \eqref{eq:phi_form} is called the \textnormal{controller margin}.
\end{definition}

Note that $\nu$ is the difference between the right-hand sides of conditions \eqref{eq:cbf_condition} and \eqref{eq:zoh_cbf_condition}, and is a bound on the discretization error that could occur between time steps. At a given state $x \in S$, a larger controller margin will necessitate a larger control input %
 to satisfy \eqref{eq:zoh_cbf_condition}. {\color{black}A sufficiently large controller margin might also necessitate inadmissible control inputs, and thus make a CBF no longer applicable to a system.} Thus, it is desired to design functions $\phi$ whose controller margins are small. For a given $T$, we call a solution $\phi_a$ less conservative than $\phi_b$ if the controller margins of $\phi_a$ and $\phi_b$ satisfy $\nu_a(T,x) \leq \nu_b(T,x), \forall x \in S$.

The controller margin is called \textit{local} (denoted as $\nu^l(T,x)$) if $\nu$ varies with $x$, and \textit{global} (denoted as $\nu^g(T)$) if $\nu$ is independent of $x$. The superscripts $l$ and $g$, respectively, denote the corresponding cases, and $\nu$ is denoted with the same sub/superscripts as the corresponding $\phi$ function. For instance, \begin{equation}
\nu_0^g(T)=\frac{l_1\Delta}{l_2}(e^{l_2T}-1) \label{eq:nu0}
\end{equation} is the controller margin of $\phi_0^g$ defined in \eqref{eq:phi0}, and is a global margin because it is independent of $x$.

Note that condition \eqref{eq:cbf_condition} imposes that the time derivative of $h$ vanishes as $h$ approaches the boundary of the safe set. In contrast, the ZOH-CBF condition \eqref{eq:zoh_cbf_condition} causes the time derivative of $h$ to vanish at a manifold in the interior of {\color{black}the} safe set. 
Inspired from this, we define a second metric of comparison, which captures the maximum distance between this manifold and the boundary of the safe set.

\begin{definition}[\textbf{Physical margin}] \label{def:physical_margin}
    For a solution $\phi$ of \ref{problem:design} with the form \eqref{eq:phi_form}, the \textnormal{physical margin} is the function $\delta:\reals_{>0}\rightarrow\reals$ defined as 
    \begin{equation}
        \delta(T) \triangleq \sup_{\color{black}\{x\in S \;\mid\; \phi(T,x) = 0\}} -h(x) \,.
    \end{equation}
\end{definition}
\noindent Intuitively, $\delta$ quantifies the effective shrinkage of the safe set due to the error introduced by discrete sampling. 
The condition \eqref{eq:zoh_cbf_condition} may exclude closed-loop trajectories from entering the set $S_\delta = \{x\; |\; -\delta \leq h(x)\leq 0\}$, while the condition \eqref{eq:cbf_condition} does not. 
A smaller physical margin $\delta$ implies a smaller set $S_\delta$ where system trajectories may not be allowed to enter.

\begin{remark} \label{rem:lower_bound}
	The physical margin $\delta$ depends on the choice of $\alpha\in\mathcal{K}$, but is always lower bounded. To capture this, define
    \begin{equation}
        \delta^{\inf}(T) \triangleq \inf_{\alpha \in \mathcal{A}} \delta(T) \,,
    \end{equation}
    where $\mathcal{A} \subseteq \mathcal{K}$ is the set of considered $\alpha$ (e.g. Lischitz continuous $\alpha$ in \ref{prior:zoh}).
    Note that $\delta^{\inf}$ may be unachievable. For instance, the $\alpha$ which yields the \textnormal{physical margin-infimum} for $\phi_0^g$ 
    is a linear function with an unbounded slope.
\end{remark}

The goal of Section~\ref{sec:methodologies} is to develop solutions to \ref{problem:design} which have lower controller and/or physical margins than $\phi_0^g$. 

\section{New Methods} \label{sec:methodologies}

This section presents three solutions to \ref{problem:design}, in both local and global forms, which follow from either continuous-time CBF conditions such as \eqref{eq:cbf_condition} (Section~\ref{sec:existing}), or discrete-time CBF conditions \cite{Discrete_CBFs,Blanchini} (Section~\ref{sec:second_order}).

\subsubsection{Extensions to Existing Literature} \label{sec:existing}

First, we note that in the proof of \ref{prior:zoh} in \cite{Cortez}, the term $\frac{\Delta}{l_2}(e^{l_2 T} - 1)$ serves as an upper bound on $||x(t) - x_k||, t \in [kT, (k+1)T)$. {\color{black} The bound is exponential, because $x_k$ is treated as a solution to a dynamical system in \cite{Cortez}. Noting that $x_k$ is a constant, the} following lemma presents an alternative upper bound.

\begin{lemma} \label{lemma:Delta_bound}
Let $\Delta = \sup_{x\in\mathcal{D},u\in U}||f(x)+g(x)u||$ where $\mathcal{D}\subseteq\reals^n$. Then for any $x_k=x(kT)\in \mathcal D$, the closed-loop trajectories of \eqref{eq:model} satisfy $||x(kT+\tau) - x_k|| \leq \tau \Delta$ for all $\tau \in \reals_{\geq0}$ such that {\color{black}$x(kT+\tau)\in\mathcal{D}$}.	
\end{lemma}

Second, we note that $\nu_0^g$ is a global margin.
The ZOH-CBF condition \eqref{eq:zoh_cbf_condition} with $\phi$ of the form \eqref{eq:phi_form} can be made less conservative by using local margins instead of global margins.
To this end, let $\mathcal{R}(x_k,T)$ denote the set of states reachable from some $x_k\in S$ in times $t\in[kT, (k+1)T)$. We are now ready to present the first main result of this paper.

\begin{theorem}\label{thm:nu1}
	Consider the set $S$ defined in \eqref{eq:safe_set} and let $\alpha\in\mathcal{K}$ be locally Lipschitz. Let $l_{L_f h}(x), l_{L_gh}(x), l_{\alpha(h)}(x)$ be the Lipschitz constants of $L_fh, L_gh,\alpha(-h)$ over the set $\mathcal{R}(x, T)$, respectively. Then the function $\phi_1^l:\reals_{>0}\times\reals^n$, defined as
	\begin{equation}
		\phi_1^l(T,x) \triangleq \alpha(-h(x)) - \underbrace{l_1(x) T \Delta(x)}_{\nu^l_1(T,x)} \,, \label{eq:phi1_local}
	\end{equation}
	solves \ref{problem:design},
	where $l_1(x) = l_{L_fh}(x)+l_{L_gh}(x)u_\textrm{max} + l_{\alpha(h)}(x)$, and $\Delta(x) = \sup_{z\in\mathcal{R}(x,T),u\in U}||f(z) + g(z) u||$. 
\end{theorem}
\begin{proof} 
For all $t\in [kT, (k+1)T)$, $k\in \mathbb N$, it holds that
\begin{align*}
	L_f h& (x) + L_g h(x) u_k =  L_f h(x_k) + L_g h(x_k) u_k -\alpha(-h(x_k)) \continue + [L_f h(x) - L_f h(x_k) + (L_g h(x) - L_g h(x_k))u_k \continue - (\alpha(-h(x)) - \alpha(-h(x_k)))] +\alpha(-h(x))\\
	& \leq  L_f h(x_k) + L_g h(x_k) u_k -\alpha(-h(x_k))+ (l_{L_fh}(x_k)+\continue + l_{L_gh}(x_k)u_\textrm{max} + l_{\alpha(h)}(x_k))||x - x_k|| + \alpha(-h(x))\\
	& \overset{\eqref{eq:phi1_local}}{\leq} L_f h(x_k) + L_g h(x_k) u_k - \phi_1^l(T,x_k) + \alpha(-h(x))\\
	& \overset{\eqref{eq:zoh_cbf_condition}}{\leq} \alpha(-h(x)),
\end{align*}
where the argument $t$ in $x(t)$ is omitted for brevity.
Thus, under \eqref{eq:zoh_cbf_condition} with $\phi = \phi_1^l$, we have that $\dot h(x(t))= L_f h(x(t)) + L_g h(x(t))u_k \leq \alpha(-h(x(t))$ for all $t\in [kT, (k+1)T)$. Since this holds for all $ k \in \mathbb{N}$, 
it follows that $\dot{h}(x(t)) \leq \alpha (-h(x(t)))$ for all $t\geq 0$. With $h(x(0))\leq 0$ and uniqueness of the closed-loop trajectories, it follows that the set $S$ is forward invariant \cite{Blanchini}, and therefore, the function $\phi_1^l$ solves \ref{problem:design}.
\end{proof}

\ref{thm:nu1} requires knowledge of the local Lipschitz constants. If these are unavailable (e.g. due to computation constraints), we can still improve upon \ref{prior:zoh} with the global margin function introduced in the following result.

\begin{corollary} \label{thm:nu1g}
    Under the assumptions of \ref{prior:zoh}, and with $l_1,\Delta$ as in \ref{prior:zoh}, the function $\phi_1^g:\reals_{>0}\times\reals^n$, defined as
    \begin{equation}
        \phi_1^g(T,x) \triangleq \alpha(-h(x)) - \underbrace{l_1 T\Delta}_{\nu_1^g(T)}  \,,  \label{eq:phi1global}
    \end{equation}
    solves \ref{problem:design}. Furthermore, for the same $\alpha$, it holds that $\nu_1^l(T,x) \leq \nu_1^g(T) < \nu_0^g(T), \forall x\in S,\forall T\in\reals_{>0}$.
\end{corollary}

\begin{proof}
    Observe that \eqref{eq:phi1_local} reduces to \eqref{eq:phi1global} for $l_1 = \sup_{x\in S} l_1(x)$ and $\Delta = \sup_{x\in S}\Delta(x)$, 
    so it holds that $\nu_1^l(T,x)\leq\nu_1^g(T), \forall x\in S, \forall T\in\reals_{>0}$ for the same $\alpha$. It follows that $\phi_1^g(T,x) \leq \phi_1^l(T,x)$. Therefore, satisfaction of \eqref{eq:zoh_cbf_condition} with $\phi_1^g$ implies satisfaction of \eqref{eq:zoh_cbf_condition} with $\phi_1^l$, and so by \ref{thm:nu1}, $\phi_1^g$ also solves \ref{problem:design}.
    
    From Taylor expansion, it holds that $T < \frac{1}{\lambda}(e^{\lambda T}-1), \forall \lambda > 0$, so it follows that $\nu_1^g(T) < \nu_0^g(T), \forall T\in\reals_{>0}$.
\end{proof}

\noindent Thus, both $\phi_1^l$ and $\phi_1^g$ reduce conservatism compared to $\phi_0^g$. 

The physical margins of $\phi_1^l$ and $\phi_1^g$ are then $\delta_1^l(T)=\alpha^{-1}(\sup_{x\in S, \phi_1^l(T,x)=0}l_1(x)T\Delta(x))$ and $\delta_1^g(T) = \alpha^{-1}(l_1T\Delta)$, respectively. 
Since $\alpha$ is assumed locally Lipschitz continuous, there exists $\Gamma \in \reals_{>0}$ and a neighborhood $Q\subseteq\reals_{\geq0}$ of the origin such that $\alpha(\lambda) \leq \Gamma \lambda, \forall \lambda \in Q$. It follows that $\alpha^{-1}(\lambda) \geq \frac{1}{\Gamma}\lambda, \forall \lambda \in Q$, so $\delta_1^l$ and $\delta_1^g$ vary linearly with $T$, as does $\delta_0^g$.

To reduce conservatism further, {\color{black}we define the following error term, inspired  by \cite{Singletary}}, representing the difference between \eqref{eq:cbf_condition} evaluated at two points $x, z\in \mathbb R^n$ for a given input $u$:
\begin{multline}
	\upsilon(x, z, u)  \triangleq L_fh(z) - L_f h(x) + (L_g h(z) - L_g h(x)) u \\
	 -\alpha(-h(z)) + \alpha(-h(x)) \,.
\end{multline}
\noindent Using this, we can state the following result.
\begin{theorem} \label{thm:nu2}
    Consider the set $S$ defined in \eqref{eq:safe_set} and let $\alpha\in\mathcal{K}$. Then the function $\phi_2^l:\reals_{>0}\times\reals^n$ as follows solves \ref{problem:design}:
    \begin{equation}
    	\phi_2^l(T,x) \triangleq \alpha(-h(x)) - \underbrace{\sup_{z\in \mathcal{R}(x,T), u\in U} \upsilon(x, z, u)}_{\nu_2^l(T,x)} \,.
    \end{equation}
\end{theorem}
\noindent The proof follows the same logic as the proof of \ref{thm:nu1}, and is omitted here in the interest of space. Using the same approach relating $\phi_1^g$ and $\phi_1^l$, we can define the function $\phi_2^g$ for which the following result can be easily shown.

\begin{corollary}
    Suppose that the conditions of \ref{thm:nu2} hold. Then the function $\phi_2^g:\reals_{>0}\times\reals^n$ as follows solves \ref{problem:design}:
    \begin{equation}
        \phi_2^g(T,x) \triangleq \alpha(-h(x)) - \underbrace{\sup_{\substack{y\in S , z\in\mathcal{R}(y,T),u\in U}} \upsilon(y,z,u)}_{\nu_2^g(T)} \,.
    \end{equation}
\end{corollary}

\begin{remark}
Using $l_1(x),\Delta(x)$ as defined in \ref{thm:nu1}, and for the same $\alpha\in\mathcal{K}$, one can show that $\upsilon(x,z,u) \leq l_1(x)T\Delta(x), \forall z\in\mathcal{R}(x,T), \forall u \in U, \forall x\in S$. Thus, for any $T\in \reals_{>0}$, the controller margins satisfy 
$\nu_2^l(T,x) \leq \nu_1^l(T,x), \forall x \in S$, 
and it follows that 
$\nu_2^g(T) \leq \nu_1^g(T)$.
\end{remark}

\subsubsection{Alternative Method Based On Second Order Dynamics} \label{sec:second_order}

The approaches discussed so far, as well as in \cite{Cortez,James,Singletary}, have relied on showing satisfaction of \eqref{eq:cbf_condition} to prove safety. In this section, rather that enforcing \eqref{eq:cbf_condition} between sample times, we start from a discrete-time CBF condition and apply it to an approximation of the continuous-time dynamics. One sufficient discrete-time CBF condition, as shown in \cite{Discrete_CBFs}, is
\begin{equation}
	h(x_{k+1}) - h(x_k) \leq -\gamma h(x_k), {\color{black}\; \forall k \in \mathbb{N}} \label{eq:discrete_cbf}
\end{equation}
for some $\gamma \in (0,1]$.
In general, this condition is not control-affine. However, its linear approximation is control-affine and thus amenable to inclusion in a QP. The error of a linear approximation of a twice differentiable function is bounded by the function's second derivative. For brevity, define 
    {\color{black}$\psi(x,u) \triangleq \nabla[\dot{h}(x)]\,(f(x) + g(x)u)$} 
which represents the second derivative of $h$ between time steps.
Since $f,g,\nabla[h]$ are assumed locally Lipschitz, $\psi$ is defined almost everywhere. {\color{black}Define the bound}

\noindent
\vspace{-8pt}\begin{equation}
	\eta(T,x) \triangleq \max\left\lbrace\left(\sup_{z\in\mathcal{R}(x,T)\setminus\mathcal{Z},u\in U} \psi(z,u)\right), 0\right\rbrace \label{eq:eta} \,,
\end{equation}
where $\mathcal{Z}$ is any set of Lebesgue measure zero {\color{black}(to account for CBFs that are not twice differentiable everywhere)}.
We are now ready to state our first solution to \ref{problem:design} that does not rely on satisfying \eqref{eq:cbf_condition} along $x(t), \forall t \geq 0$.

\begin{theorem} \label{thm:nu3}
    The function $\phi_3^l:\reals_{>0}\times\reals^n$, defined as
	\begin{equation}
		\phi_3^l(T,x) \triangleq -\frac{\gamma}{T}h(x) - \underbrace{\frac{1}{2}T\eta(T,x)}_{\nu_3^l(T,x)} \, \label{eq:phi3}
	\end{equation}
	solves \ref{problem:design}, for any $\gamma \in (0,1]$.
\end{theorem}
\begin{proof}
	For $k\in \mathbb N$, let $t = kT + \tau, \tau \in [0, T]$. For any $x(t)\in\mathcal{R}(x_k,T)$, $u$ as in \eqref{eq:piecewise_control} with $u_k \in U$, {\color{black}the time derivative $\dot{h}(x(t))=\dot{h}(x(t_k)) +  \int_{kT}^{kT+\tau} \ddot{h}(x(\sigma)) d\sigma=\dot{h}(x(t_k)) +  \int_{kT}^{kT+\tau} \psi(x(\sigma), u_k) d\sigma$} satisfies
	\begin{align*}
	    \dot{h}(x(t)) 
	    %
	    %
	    \overset{\eqref{eq:eta}}{\leq} \dot{h}(x_k) + \tau \eta(T,x_k).
	\end{align*}
\newcommand{\compact}{}
    Similarly, {\color{black}$h(x(t))=h(x(t_k)) +  \int_{kT}^{kT+\tau} \dot{h}(x(\sigma)) d\sigma$} satisfies
    \begin{align*} 
        h(x(t)) 
        %
        %
        %
        \leq& \ h(x_k) + \dot{h}(x_k) \tau + \frac{1}{2}\tau^2 \eta(T,x_k) \\
        \overset{\eqref{eq:phi3}}{\leq}& \compact \  h(x_k) + \phi_3^l(T,x)\tau + \frac{1}{2}\tau^2 \eta(T,x_k)
        \\ =& \compact \ h(x_k) - \frac{\gamma\tau}{T}h(x_k) - \frac{\tau T}{2}\eta(T,x_k) + \frac{\tau^2}{2} \eta(T,x_k) \\
        =& \compact \left(1 - \frac{\gamma \tau}{T}\right) h(x_k) + \frac{\tau}{2}\eta(T,x_k) \left( \tau - T \right).
    \end{align*}
	By definition in \eqref{eq:eta}, $\eta(T,x_k) \geq 0$. Suppose $h(x_k) \leq 0$. Then both terms of the above equation are nonpositive for any $\tau\in [0,T]$, so $h(x(t)) \leq 0, \forall t\in[kT, kT+T]$, and thus, $h(x_{k+1}) \leq 0$. Hence, given $x(0)$ such that $h(x(0)) \leq 0$ and applying \eqref{eq:zoh_cbf_condition} at every time step with $\phi=\phi_3^l$, it follows by induction that $h(x(t))\leq 0, \forall t \geq 0$, and thus $S$ is forward invariant along the closed-loop trajectories of \eqref{eq:model}. Therefore, the function $\phi_3^l$ solves \ref{problem:design}.
\end{proof}

Similar to the previous cases, we can define the global version $\phi_3^g$ as follows and show that it also solves \ref{problem:design}.
\begin{corollary}
    Under the assumptions of \ref{thm:nu3}, the function $\phi_3^g:\reals_{>0}\times\reals^n$ as follows solves \ref{problem:design}:
    \begin{equation}
        \phi_3^g (T,x) \triangleq -\frac{\gamma}{T}h(x) - \underbrace{\frac{1}{2}T\sup_{z\in S}\eta(T,z)}_{\nu_3^g(T)} \,. \label{eq:phi3g}
    \end{equation}
\end{corollary}

We now study how the solutions $\phi_3^l,\phi_3^g$ compare to prior methods, by first comparing the controller margins as follows.

\begin{theorem} \label{thm:comparison}
	Under the assumptions of \ref{thm:nu3}, the controller margins for $\phi_3^l,\phi_3^g$ and $\phi_1^l,\phi_1^g$ satisfy $\nu_3^l(T,x) \leq \frac{1}{2}\nu_1^l(T,x)$ and $\nu_3^g(T)\leq \frac{1}{2}\nu_1^g(T)$, $\forall x \in S, \forall T\in \reals_{>0}$.
\end{theorem}
\begin{proof}
	Since $f$ and $g$ are differentiable almost everywhere, their Lipschitz constants are the norms of their gradients. Denote $\mathcal{X}(x) = \mathcal{R}(x,T)\setminus\mathcal{Z}\times U$. Thus, 
	
	\noindent\vspace*{-14pt}
	\begin{align*}
	    \nu_3^l(T&,x) = \frac{T}{2}\eta(T,x) = \frac{T}{2}\max \left\lbrace \sup_{(z,u)\in\mathcal{X}(x)} \psi(z,u), 0 \right\rbrace \\
		=& \frac{T}{2}\max \left\lbrace \sup_{(z,u)\in\mathcal{X}(x)} \nabla[L_f h(z) + L_g h(z) u] \dot{z}, 0 \right\rbrace \\
		\leq& \frac{T}{2} \sup_{(z,u)\in\mathcal{X}(x)} \left( \norm{\nabla[L_f h(z)]} + \norm{\nabla[L_g h(z)]}u_\textrm{max} \right) ||\dot{z}|| \\
		\leq& \frac{T}{2}\left( l_{L_f h}(x) + l_{L_g h}(x)u_\textrm{max} \right) \Delta(x) \\
		=& \frac{1}{2}\nu_1^l(T,x) - \frac{1}{2} l_{\alpha(h)}(x) T \Delta(x) \hspace{1.15in}
	\end{align*}
	The inequality for the global margins follows immediately.
\end{proof}

Thus, solutions $\phi_1^g, \phi_2^g$ are provably less conservative than the existing solution $\phi_0^g$, and $\phi_3^g$ is provably half as conservative as $\phi_1^g$ (and similarly for the local margins). It is difficult to analytically compare $\phi_2^l,\phi_2^g$ with $\phi_3^l,\phi_3^g$, so we address this via simulations in Section~\ref{sec:results}.

Lastly,
we consider the physical margins. Since $\alpha\in\mathcal{K}$ from \eqref{eq:phi_form} is specified as $\alpha(\lambda) = \frac{\gamma}{T}\lambda$ in \eqref{eq:phi3},\eqref{eq:phi3g}, the physical margin of $\phi_3^g$ is $\delta_3^g(T) = \frac{T}{\gamma}\nu_3^g(T) = \frac{T^2}{2\gamma} \sup_{x\in S\setminus\mathcal{Z}, u\in U}\psi(x,u)$, and similarly $\delta_3^l(T) = \frac{T^2}{2\gamma} \sup_{x\in S\setminus\mathcal{Z},\phi_3^l(T,x)=0, u\in U}\psi(x,u)$. This implies $\delta_3^l,\delta_3^g$ vary quadratically with $T$, while $\delta_0^g,\delta_1^l,\delta_1^g$ vary only linearly with $T$. 
Note that choosing $\alpha(\lambda)=\frac{\gamma}{T}\lambda$ does not similarly reduce $\delta_0^g,\delta_1^l,\delta_1^g,\delta_2^l,\delta_2^g$, because $l_{\alpha(h)}$ 
would increase inversely with $T$.
Thus, reducing step size is far more effective at reducing physical margin when $\phi_3^l$ or $\phi_3^g$ is used.

\section{Simulation Results} \label{sec:results}

We implemented the methods in Section~\ref{sec:methodologies} on two systems. First, we tested the unicycle system, described by
\begin{equation*}
\dot x_1 = u_1 \cos(x_3), \; \dot x_2 = u_1 \sin (x_3), \; \dot x_3 = u_2,
\end{equation*}
where $[x_1,\,x_2]\transpose$ is the position, $x_3$ is the orientation, {\color{black}and $u_1$,$u_2$ are the linear and} angular velocity of the agent; {\color{black}its task was to move around an obstacle} at the origin using the CBF \cite{polar_cbf}
\begin{equation*}
	h = \rho - \sqrt{x_1^2 + x_2^2 - ({\textstyle\wrap_\pi}(x_3 - \sigma \arctan2(x_2, x_1)))^2} \,,
\end{equation*}
where $\rho$ is the radius to be avoided, and $\sigma$ is a shape parameter.
Second, we tested a spacecraft pointing system, described by
\begin{equation*}
\dot p = \omega\times p, \; \dot \omega = u,
\end{equation*}
where $p\in\reals^3, ||p||\equiv 1$, is a pointing vector, $\omega\in\reals^3$ is the angular velocity, and $u\in \reals^3$ is the angular acceleration.
The system was tasked with reorienting an instrument while pointing away from an inertially-fixed {\color{black}vector} using the CBF
\begin{equation*}
	h = s \cdot p - \cos(\theta) + \mu (s \cdot (\omega \times p)) | s \cdot (\omega \times p) | \,,
\end{equation*}
where $s\in\reals^3$, $||s|| = 1$, is a constant vector pointing to {\color{black}an} object to be avoided, $\theta$ is the smallest allowable angle, and $\mu$ is a shape parameter. We also constrained $||\omega||_\infty \leq 0.2$, because otherwise the global controller margins are unbounded. 

\begin{table}
	\centering
	\vspace{2pt}
	\begin{tabular}{r | c | c}
		Parameter & Unicycle & Spacecraft \\ \hline
		Exclusion Zone & $\rho = 10$ & $\theta = \pi/5$ \\
		Shape Parameter & $\sigma = 1$ & $\mu = 100$ \\
		$U$ & $\begin{aligned} u_1 &\in [0,5]\\ u_2 &\in [-0.25, 0.25] \end{aligned}$ & $||u||_\infty \leq 0.01$ \\
		\hdashline
		$\nu_0^{g}(0.1)$ & $1.316(10)^{50}$ & 14.20 \\
		$\nu_1^{g}(0.1)$ & 570.3 & 2.946 \\
		$\nu_2^{g}(0.1)$ & 0.6908 & 0.8815 \\
		$\nu_3^{g}(0.1)$ & 0.1319 & 0.1194 
	\end{tabular}
	\caption{\small Simulation parameters and global controller margins}
	\label{tab:params}
\end{table}

\begin{table}[]
    \centering
    \resizebox{1\columnwidth}{!}{
\begin{tabular}{c| c| c | c |c |c |c }  & \multicolumn{3}{c|}{\normalsize Unicycle}& \multicolumn{3}{c}{\normalsize Spacecraft}\\ \hline
$T$ & 0.1 & 0.01 & 0.001 & 0.1 & 0.01 & 0.001 \\ 
\hline 
$\delta_0^{g,\inf}$ & $1.2(10)^{42}$ & 420 & 0.010 & 9.8 &  0.23  & 0.021\\ 
$\delta_1^{g,\inf}$  & 0.54 & 0.054 & 0.0054 & 2.0 &  0.20 & 0.020\\ 
$\delta_2^{g,\inf}$ & 0.53 & 0.053 & 0.0053&  0.81 &  0.082 & 0.0082 \\ 
$\delta_3^{g,\inf}$ & 0.013 &  $1.3(10)^{-4}$  &  $1.3(10)^{-6}$ & 0.013 &  $1.3(10)^{-4}$  &  $1.3(10)^{-6}$
\end{tabular}}
    \caption{\small Global physical margins for selected time-steps $T$}
    \label{tab:physical}
\end{table}

Both systems were tested for $T = 0.1$. For functions $\phi_0^g,\phi_1^l,\phi_1^g,\phi_2^l,\phi_2^g$, we used $\alpha(\lambda) = \lambda$, and for $\phi_3^l,\phi_3^g$, we used $\gamma = 1$. Notable parameters and the controller margins for the selected time-step for both systems are listed in Table~\ref{tab:params}. The physical margins for various time-steps are listed in Table~\ref{tab:physical}. Note that $\delta_3^{g,\inf}$ is less than $\delta_0^{g,\inf},\delta_1^{g,\inf},\delta_2^{g,\inf}$, which means that $\phi_3^l$ and $\phi_3^g$ will allow the system trajectories to get closer to the boundary of the safe set than any of the other methods. 
Moreover, for the smaller values of $T$ in Table~\ref{tab:physical}, $\delta_3^{g,\inf}$ varies quadratically with $T$, while $\delta_0^{g,\inf}, \delta_1^{g,\inf}, \delta_2^{g,\inf}$ vary linearly with $T$.
The agents used a controller of the form
\begin{equation}
	u = \argmin_{u\in K_\textrm{zoh}} ||u - u_\textrm{nom}|| \label{eq:the_qp}
\end{equation}
where $u_\textrm{nom}$ is a nominal control law that ignores the obstacle, and $K_\textrm{zoh} \subseteq U$ is the set of control inputs satisfying \eqref{eq:zoh_cbf_condition}.

For the unicycle agent, the exact reachable sets $\mathcal{R}(x_k,T)$ were computed, and $\nu_1^l,\nu_2^l,\nu_3^l$ were computed using online maximizations of $l_1(x), \Delta(x), \upsilon(x, z, u), \psi(x,u)$ over these sets. 
For the spacecraft system {\color{black}(and in general for nonlinear systems)}, these reachable sets are harder to compute online, so we note that all preceding results still hold when $\mathcal{R}(x_k,T)$ is replaced with any superset of $\mathcal{R}(x_k,T)$ (though this {\color{black}in principle} increases conservatism). Also, by \ref{lemma:Delta_bound}, $\mathcal{R}(x_k,T) \subseteq B_{T\Delta}(x_k)$. To this end, given Lipschitz constants $l_f, l_g$ for functions $f,g$, respectively, an upper bound for $\Delta$ is
\begin{equation}
	\Delta_0(x_k) \triangleq \frac{||f(x_k)|| + ||g(x_k)||u_\textrm{max}}{1 - (l_f + l_g u_\textrm{max})T} \,, \label{eq:delta}
\end{equation}
assuming that 
the denominator of \eqref{eq:delta} is positive. Thus, the margins $\nu_1^l$,$\nu_2^l$,$\nu_3^l$ for the spacecraft were computed using online maximizations over the superset $B_{T\Delta_0(x_k)}(x_k)$.
{\color{black} These maximizations took approximately 0.028, 0.026, and 0.018 seconds for $\nu_1^l,\nu_2^l,\nu_3^l$, respectively, for the unicycle, and 0.058, 0.071, and 0.045 seconds, respectively, for the spacecraft on a 3.5 GHz computer using MATLAB R2019b. For higher-dimensional systems, these online computations could limit the applications of the local methods. Each global margin took under a minute to compute. }
We then computed the states using the exact dynamics, 
and solved \eqref{eq:the_qp} using OSQP \cite{osqp}. 
In total, 7 {\color{black}solutions} ($\phi_0^g$,$\phi_1^g$,$\phi_1^l$,$\phi_2^g$,$\phi_2^l$,$\phi_3^g$,$\phi_3^l$) {\color{black}to} \ref{problem:design}
were tested\footnote{Simulation code may be found at \href{https://github.com/jbreeden-um/phd-code/tree/main/2021/L-CSS CBFs for Sampled Data Systems}{https://github.com/jbreeden-um/phd-code/tree/main/2021/L-CSS CBFs for Sampled Data Systems}}.

\begin{figure}
	\centering
	\begin{tikzpicture}
        \node[anchor=south west,inner sep=0] (image) at (0,0) {\includegraphics[width=0.9\columnwidth,trim={0.25in, 0in, 0.5in, 0.3in},clip]{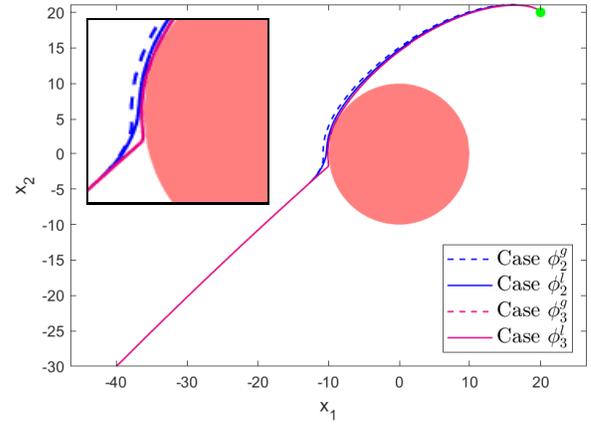}}; %
        \node[anchor=south west,inner sep=0] (image) at (1.1,2.92) 
        {\includegraphics[width=0.94in,trim={3in, 2.1in, 2.5in, 1.3in},clip,frame]{Uni_Trajectories.eps}
        };
    \end{tikzpicture}
	\caption{\small{The trajectories of the unicycle for 4 of the margin functions}}
	\label{fig:unicycle_trajectories}
\end{figure}
\begin{figure}
	\centering
	\includegraphics[width=0.32\textwidth,trim={1in, 0.2in, 1in, 0.2in},clip]
	{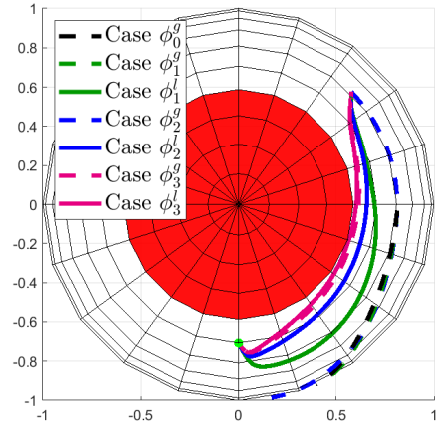}
	\caption{\small{The trajectories of the spacecraft for all 7 margin functions}}
	\label{fig:spacecraft_trajectories}
\end{figure}

The trajectories for the two systems are plotted in Figs.~\ref{fig:unicycle_trajectories}-\ref{fig:spacecraft_trajectories}, {\color{black}where} the green markers are the target locations. As expected, certain methods took wider arcs around the obstacles than others based on the relative values of $\nu$ and $\delta$.
For the unicycle, only four methods are shown because using $\phi_0^g$,$\phi_1^g$,$\phi_1^l$ resulted in the agent turning away from the target. Similarly for the spacecraft, using $\phi_0^g$,$\phi_1^g$,$\phi_2^g$ eventually resulted in divergence from the target attitude {\color{black} as the QP was unable to satisfy \eqref{eq:zoh_cbf_condition}}.

\begin{figure}
	\centering
	\includegraphics[width=0.968\columnwidth]{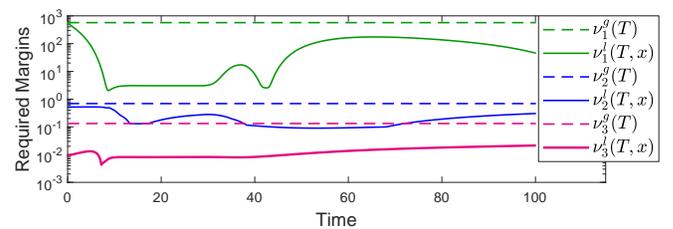}
	\caption{
\small Controller margins for the unicycle system
	}
	\label{fig:unicycle_margins}
\end{figure}

\begin{figure}
	\centering
	\includegraphics[width=0.968\columnwidth]{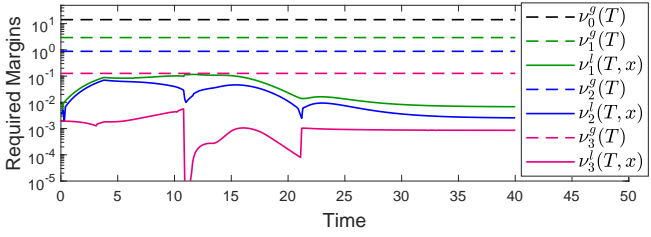}
	\caption{
	\small Controller margins for the spacecraft system}
	\label{fig:spacecraft_margins}
\end{figure}

The instantaneously required controller margins $\nu$ for every method, computed for $x(t)$ along the $\phi_3^l$ trajectories from Figs.~\ref{fig:unicycle_trajectories}-\ref{fig:spacecraft_trajectories}, are plotted in Figs.~\ref{fig:unicycle_margins}-\ref{fig:spacecraft_margins}. As predicted by \ref{thm:comparison}, the green solid and dashed lines for controller margins $\nu_1^l,\nu_1^g$ are always at least double 
(and generally an order of magnitude greater than)
the equivalent pink lines for $\nu_3^l,\nu_3^g$, respectively. The controller margins $\nu_2^l,\nu_2^g$ were also always larger than $\nu_3^l,\nu_3^g$, though this is not guaranteed by \ref{thm:comparison}. Interestingly, for the unicycle, the global margin $\nu_3^g$ was generally similar to or smaller than the local margin $\nu_2^l$, whereas for the spacecraft, $\nu_3^g$ was larger than both $\nu_1^l$ and $\nu_2^l$. However, the trajectories corresponding to $\phi_3^g$ still approached closer to the obstacles than those under $\phi^l_1$ and $\phi_2^l$ in both Figs.~\ref{fig:unicycle_trajectories}-\ref{fig:spacecraft_trajectories} because $\phi_3^g$ has an order of magnitude smaller physical margin.

\begin{figure}
	\centering
	\includegraphics[width=0.968\columnwidth]{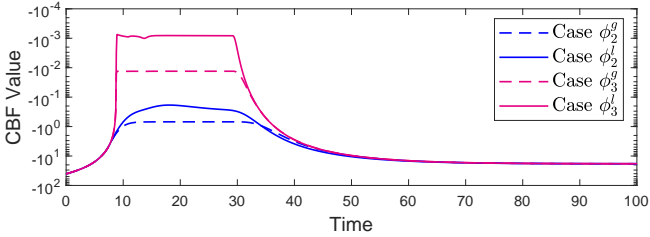}
	\caption{\small{CBF values along the 4 unicycle trajectories in Fig.~\ref{fig:unicycle_trajectories}}}
	\label{fig:unicycle_barriers}
\end{figure}
\begin{figure}[t]
	\centering
	\includegraphics[width=0.968\columnwidth]{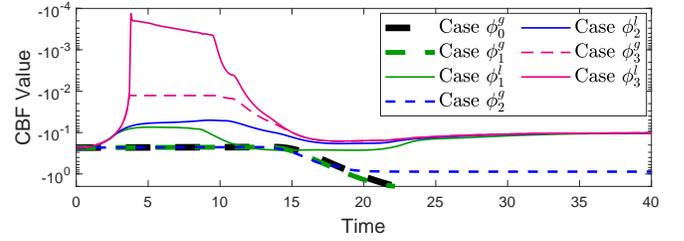}
	\caption{\small{CBF values along the 7 spacecraft trajectories in Fig.~\ref{fig:spacecraft_trajectories}}}
	\label{fig:space_barriers}
\end{figure}

Finally, the CBF values during every simulation are shown in Figs.~\ref{fig:unicycle_barriers}-\ref{fig:space_barriers}. From this, we see that the trajectories corresponding to $\phi_3^l$ and $\phi_3^g$ come within an order of magnitude closer to the boundary than those for any of the other methods. The dashed lines in Figs.~\ref{fig:unicycle_barriers}-\ref{fig:space_barriers} also agree with the theoretical physical margins listed in Table~\ref{tab:physical}.

Noting these physical margins, we added a second constraint to the unicycle system that forced the unicycle to navigate through a narrow corridor only 0.3 units wide, shown in Fig.~\ref{fig:two_obstacles}. The unicycle operating under $\phi_3^g$ or $\phi_3^l$ made it through the obstacles, while the best of the other methods ($\phi_2^l$) could not.

\begin{figure}[h!]
    \centering
    \begin{tikzpicture}
        \node[anchor=south west,inner sep=0] (image) at (0,0) {\includegraphics[width=0.9\columnwidth,trim={0.25in, 0in, 0.5in, 0.34in},clip]{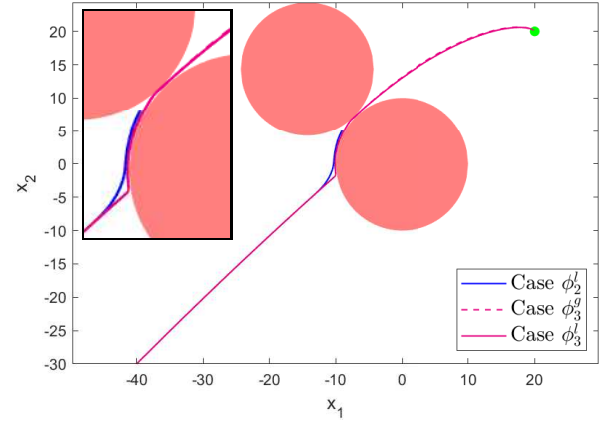}};
        \node[anchor=south west,inner sep=0] (image) at (0.95,2.4) {\includegraphics[width=0.77in,trim={3in, 2in, 2.5in, 1in},clip,frame]
        {Uni_TwoObstacles.eps}
        };
    \end{tikzpicture}
    \caption{\small{A simulation with two tightly-spaced obstacles, in which controllers using margins $\phi_3^l$ and $\phi_3^g$ permit passage through the obstacles, while the other functions force the agent to stop.}} 
    \label{fig:two_obstacles}
\end{figure}

\section{Conclusions}

We presented new conditions for ensuring safety in sampled-data systems that provably reduce conservatism compared to earlier results. We introduced two metrics for quantifying the margin in both the control input and in the effective shrinkage of the safe set. We showed that the proposed conditions have smaller margins compared to those in earlier studies, and demonstrated the improved performance of the proposed results via numerical case studies. In particular, the physical margin of the last condition proposed varied quadratically with the discretization time-step, while that of the existing approaches varied linearly. This allowed completion of objectives that were not possible using other methods under the same time-step. Future work includes studying whether higher-order approximations can further decrease conservatism.

\bibliographystyle{ieeetran}
\bibliography{sources}
\end{document}